\documentclass[10pt]{article}
\usepackage{graphicx,float,latexsym,amssymb,subeqnarray,amsmath,amsfonts}

\def\re{{\Re}}
\def\R{\mathbb{R}}
\def\C{\mathbb{C}}

\newtheorem{theorem}{Theorem}

\newtheorem{lemma}[theorem]{Lemma}

\newenvironment{proof}{\noindent {\it Proof}~}{}

\title{
\mbox{Stability of the Modified Craig--Sneyd scheme for}
\mbox{two-dimensional convection-diffusion equations} 
\mbox{with mixed derivative term}}

\author{
K.~J.~in 't Hout\thanks{Department of Mathematics and Computer Science,
University of Antwerp, Middelheimlaan 1, 2020 Antwerp, Belgium
(e-mail: \texttt{karel.inthout@ua.ac.be}).}
~and
C. Mishra\thanks{Department of Mathematics and Computer Science,
University of Antwerp, Middelheimlaan 1, 2020 Antwerp, Belgium
(e-mail: \texttt{chittaranjan.mishra@ua.ac.be}).}
}
\date{April 7, 2010}

\begin{document}
\maketitle

\begin{abstract}
The Modified Craig--Sneyd (MCS) scheme is a promising splitting scheme of the 
ADI type introduced by In 't Hout \& Welfert [{\it Appl. Num. Math.~\textbf{59} 
(2009)}] for multi-dimensional pure diffusion equations having mixed 
spatial-derivative terms.
In this paper we investigate the extension of the MCS scheme to two-dimensional 
convection-diffusion equations with a mixed derivative.
Both necessary and sufficient conditions on the parameter $\theta$ of the scheme 
are derived concerning unconditional stability in the von Neumann sense.
\end{abstract}
\vspace{0.2cm}

\vspace{3mm}

\vfill\eject

\newpage

\setcounter{equation}{0}
\section{Introduction}\label{intro}
We consider the numerical solution of initial value problems for large 
systems of ordinary differential equations (ODEs),
\begin{equation}\label{ODE}
U'(t) = F(t,U(t)) \quad (t\ge 0),\quad U(0)=U_0,
\end{equation}
with given  vector-valued function $F$, given initial vector $U_0$, 
and unknown vectors $U(t)$ (for $t>0$).
Our interest in this paper lies in systems (\ref{ODE}) that arise from
semi-discretization of initial-boundary value problems for two-dimensional 
convection-diffusion equations possessing a mixed spatial-derivative term,
\begin{equation}\label{PDE}
\frac{\partial u}{\partial t} = 
d_{11} u_{xx}+(d_{12}+d_{21}) u_{xy}+d_{22} u_{yy}+c_1 u_x+c_2 u_y.
\end{equation}
Here $c=(c_i)$ and $D=(d_{ij})$ denote a given real vector and a given 
positive semi-definite real matrix, respectively.
A main application area of equations of the kind (\ref{PDE}) is financial 
option pricing theory, where mixed derivative terms $u_{xy}$ arise 
naturally since the underlying Brownian motions are usually correlated 
to each other.
Extensive details and examples of financial applications are given in, 
for example, the references \cite{S04,TR00,W99}.

For the numerical solution of semi-discrete problems (\ref{ODE}), splitting 
schemes form an effective and popular means, cf.~e.g.~\cite{HV03,MiGr80}.
This paper is devoted to the analysis of a recent splitting scheme of the 
Alternating Direction Implicit (ADI) type that has been tailored so as to 
deal with equations possessing a mixed derivative term.
Let $\theta >0$ be a given fixed parameter.
Assume the right-hand side function $F$ is decomposed into a sum
\begin{equation}\label{splitting}
F(t,v) = F_0(t,v) + F_1(t,v) + F_2(t,v),
\end{equation}
where $F_0$ represents the contribution to $F$ stemming from the mixed derivative 
term, and $F_j$ (for $j= 1,2$) represents the contribution to $F$ stemming from
all spatial derivative terms in the $j$-th spatial direction.
Let $\Delta t >0$ be a given time step and define temporal grid points by
$t_n= n\cdot \Delta t$ ($n=0,1,2,\ldots$).
We consider the following splitting scheme for (\ref{ODE}), generating in
a one-step fashion successive approximations $U_1, U_2, U_3,\ldots $ to 
$U(t_1), U(t_2), U(t_3), \ldots$\,:
\vskip0.5cm\noindent
\begin{equation}\label{MCS}
\left\{\begin{array}{l}
Y_0 = U_{n-1}+\Delta t\, F(t_{n-1},U_{n-1}), \\ \\
Y_j = Y_{j-1}+\theta\Delta t \left(F_j(t_n,Y_j)-F_j(t_{n-1},U_{n-1})\right)~~~(j=1,2),\\ \\
\widehat{Y}_0 = Y_0+ \theta \Delta t \left(F_0(t_n,Y_2)-F_0(t_{n-1},U_{n-1})\right),\\ \\
\widetilde{Y}_0 = \widehat{Y}_0+ (\frac{1}{2}-\theta )\Delta t \left(F(t_n,Y_2)-F(t_{n-1},U_{n-1})\right), \\ \\
\widetilde{Y}_j = \widetilde{Y}_{j-1}+\theta\Delta t \,(F_j(t_n,\widetilde{Y}_j)-F_j(t_{n-1},U_{n-1}))~~~(j=1,2),\\ \\
U_n = \widetilde{Y}_2.
\end{array}\right.
\end{equation}
Method (\ref{MCS}) is called the {\it Modified Craig--Sneyd (MCS) scheme}.
It has recently been introduced, in a slightly more general form, by In 't Hout 
\& Welfert \cite{IHW09}.
Taylor expansion yields that the MCS scheme has classical order of consistency 
equal to two for any value~$\theta$.

The MCS scheme can be viewed as an extension of the second-order Craig--Sneyd 
(CS) scheme proposed in \cite{CS88}.
The latter scheme, called ``iterated scheme'' in loc.~cit., is equivalent to 
(\ref{MCS}) with parameter value $\theta = \frac{1}{2}$.

A perusal of (\ref{MCS}) shows that the $F_0$ term is always treated {\it explicitly}, 
whereas the $F_1$ and $F_2$ terms are treated {\it implicitly}.
More precisely, the MCS scheme starts with an explicit Euler step applied to
the full system (\ref{ODE}) which is succeeded by two implicit corrections 
corresponding to each of the two spatial directions. 
Subsequently, an explicit update is performed, which is followed again by 
two implicit, unidirectional corrections.
Accordingly, the MCS scheme retains the well-known key advantage of ADI schemes 
over standard implicit methods, such as the Crank--Nicolson scheme, that the 
(linear or nonlinear) systems to be solved in each time step are much easier 
to handle.

The adaptation of ADI schemes to convection-diffusion equations with 
mixed derivative terms has been studied by a number of authors. 
Several stability results, in the sense of von Neumann, have been obtained.
McKee et al.~\cite{MM70,MWW96} considered a simpler version of (\ref{MCS}),
which is equivalent to the first two lines with $U_n = Y_2$. 
This basic scheme, also known as the Douglas scheme, is of order one for any 
value $\theta$ in the presence of a mixed derivative term.
McKee et al.~showed that if $\theta = \frac{1}{2}$, then it is unconditionally 
stable when applied to a standard finite difference (FD) discretization of 
(\ref{PDE}).
Next, Craig \& Sneyd \cite{CS88} formulated the second-order CS scheme and 
proved that this scheme is unconditionally stable in the case of (\ref{PDE}) 
with $c\equiv \mathbf{0}$.
Recently In 't Hout \& Welfert \cite{IHW07,IHW09} extended the above
stability results in various ways. 
We state here the main results pertinent to the situation at hand.
Firstly, for the CS scheme unconditional stability was proved \cite{IHW07} 
in the general case of (\ref{PDE}).
Secondly, it was shown \cite{IHW09} that in the case of (\ref{PDE}) with
$c\equiv \mathbf{0}$ the MCS scheme (\ref{MCS}) is unconditionally stable 
whenever $\theta \ge \frac{1}{3}$.

Up to now it is an important open question when the MCS scheme, with 
$\theta \not= \frac{1}{2}$\,, is unconditionally stable in the application 
to general equations (\ref{PDE}), i.e., with arbitrary $c$ and positive 
semi-definite $D$.
As it turns out, an analysis of this is not straightforward, related to the 
fact that the eigenvalues of the semi-discrete linear operators move from 
the real line in the pure diffusion case to the complex plane in the general, 
convection-diffusion case.
In the present paper we shall arrive at positive results on the above question.

For the stability analysis we consider the linear scalar test equation
\begin{equation}\label{test}
 U'(t) = (\lambda_0+\lambda_1+\lambda_2)U(t)
\end{equation}
with complex constants $\lambda_j$ ($0\le j\le 2$).
When applied to (\ref{test}), the MCS scheme (\ref{MCS}) reduces to the 
scalar iteration
\begin{equation}\label{unS}
U_n = S_\theta (z_0,z_1,z_2)\,U_{n-1}
\end{equation}
with $z_j = \Delta t\cdot \lambda_j$ \,($0\le j\le 2$) and
\begin{equation}\label{S}
S_\theta (z_0,z_1,z_2) =
1+\frac{z_0+z}{p}+\theta\,\frac{z_0(z_0+z)}{p^2}+
(\tfrac{1}{2}-\theta)\,\frac{(z_0+z)^2}{p^2}\,,
\end{equation}
where we use the notation
\begin{equation*}\label{zp}
z = z_1+z_2 \quad\text{and}\quad p = (1-\theta z_1)(1-\theta z_2).
\end{equation*}
The iteration (\ref{unS}) is stable if
\begin{equation}
\label{stabcondS}
|S_\theta (z_0,z_1,z_2)|\le 1.
\end{equation}

In the von Neumann framework, the $\lambda_j$ represent eigenvalues of
the linear operators $F_j$ that are obtained after semi-discretization,
on a uniform spatial grid, of the convection-diffusion equation (\ref{PDE}) 
with constant coefficients and periodic boundary condition. 
Corresponding to the positive semi-definiteness of the diffusion matrix $D$, 
it was shown in \cite{IHW07} (cf.~also Sect.~\ref{Conclusion}) that for 
standard FD discretizations the following condition on the scaled 
eigenvalues $z_j$ is fulfilled,
\begin{equation}
\label{cond}
|z_0|\le 2\sqrt{\Re z_1\Re z_2}\,,~~\Re z_1\le 0,~~\Re z_2\le 0,
\end{equation}
where all bounds are sharp.
In view of this, a natural stability requirement on the scheme (\ref{MCS}) 
when applied to equations (\ref{PDE}) with mixed derivative terms is 
that $(\ref{stabcondS})$ holds whenever $(\ref{cond})$ is satisfied.

An outline of the rest of this paper is as follows.
In Sect.~\ref{Stabresults} we study for which parameter values 
$\theta$ the implication $(\ref{cond}) \Rightarrow (\ref{stabcondS})$ 
is fulfilled.
Four cases are investigated, depending on whether $z_0$ is real or 
complex valued and whether $z_1, z_2$ are (both) real or complex 
valued.
In Sect.~\ref{Conclusion} the results of Sect.~\ref{Stabresults}
are applied and discussed relevant to an application of the MCS
scheme (\ref{MCS}) to (\ref{PDE}).

\section{Stability results for the MCS scheme}\label{Stabresults}
Let $I$ denote the imaginary unit.
In this section we study the stability requirement
$(\ref{cond}) \Rightarrow (\ref{stabcondS})$.
The following introductory result gives a criterion on $\theta$ for 
the case $z_0=0$.
This is pertinent to the situation where no mixed derivative term is 
present in (\ref{PDE}).

\begin{theorem}\label{Thm1}
There holds $|S_\theta (0,z_1,z_2)| \le 1$ for all $z_1, z_2 \in \C$ 
with $\Re z_1 \le 0$, $\Re z_2 \le 0$ if and only if 
$\theta \ge \tfrac{1}{4}$.
\end{theorem}

\begin{proof}
The rational function $S_\theta (0, z_1, z_2)$ has no poles in the set 
$\Re z_1, \Re z_2 \le 0$
and therefore attains its maximum on the boundary of this set.
Thus assume $z_1 = I b_1$, $z_2 = I b_2$ with $b_1, b_2 \in \mathbb{R}$.
We have
\begin{equation*}
|S_\theta (0, z_1, z_2)| \leq 1  ~~\Longleftrightarrow~~
|p^2 + pz + (\tfrac{1}{2}-\theta)z^2|^2 - |p^2|^2 \leq 0.
\end{equation*}
Write $u = 1-\theta^2b_1b_2$ and $v = b_1+b_2$.
Then
\begin{equation*}
p = u - \theta I v ~,~
p^2 = u^2 - \theta^2 v^2 -  2 \theta I uv ~,~ pz = \theta v^2 + I uv ~,~ z^2 = -v^2
\end{equation*}
and it follows after some algebraic manipulations that
\begin{equation*}
|p^2 + pz + (\tfrac{1}{2}-\theta)z^2|^2 - |p^2|^2 =
\left[ \left( \theta^2-2\theta+\tfrac{1}{2} \right)^2 - \theta^4 \right] v^4.
\end{equation*}
Hence, 
\[
|S_\theta (0,z_1,z_2)| \le 1 \quad \textrm{whenever } \Re z_1, \Re z_2 \le 0 
\]
if and only if
\begin{equation*}
| \theta^2-2\theta+\tfrac{1}{2} | \le \theta^2,
\end{equation*}
which is equivalent to $\theta \ge \tfrac{1}{4}$.
\quad\mbox{\tiny $\blacksquare$}
\end{proof}
\vskip0.3cm

In \cite{IHW09} the stability of ADI schemes for pure diffusion 
equations with mixed derivatives was analyzed.
This concerns the case where all $z_j$ are real-valued.
For the MCS scheme and two spatial dimensions, the following criterion 
on $\theta$ was obtained.

\begin{theorem}\label{Thm2}
There holds $|S_\theta (z_0,z_1,z_2)| \le 1$ whenever $z_0, z_1, z_2 \in \R$ 
satisfy $(\ref{cond})$ if and only if $\theta \ge \tfrac{1}{3}$.
\end{theorem}
\begin{proof}
See \cite[Thm.~2.5]{IHW09}.
\quad\mbox{\tiny $\blacksquare$}
\end{proof}
\vskip0.3cm
In most applications, also a convection term is present.
Accordingly, one is led to considering complex-valued $z_1, z_2$.
The next theorem gives a necessary condition on $\theta$ for this 
situation.

\begin{theorem}\label{Thm3}
Suppose $|S_\theta (z_0,z_1,z_2)| \le 1$ for all $z_0\in \R$ and 
$z_1, z_2\in \C$ satisfying $(\ref{cond})$.
Then $\theta \geq \frac{2}{5}$.
\end{theorem}
\begin{proof}
The result is obtained by a Taylor expansion at the point $z_0=z_1=z_2=0$. 
We take $z_0 = -2a$ and $z_1=z_2=a \eta$ where $\eta=1+I$ and $a\in\R$ with 
$a\uparrow 0$. 
This choice was found to be convenient after numerical experimentation. 
Inserting into (\ref{S}) and using $1/(1-\xi) = 1+\xi+\xi^2+{\cal O}(\xi^3)$ 
($\xi\rightarrow 0$), 
it follows that
\begin{eqnarray*}
  S_\theta (-2a,a \eta,a \eta)  
  &=& 1 + \frac{2 I a}{(1-\theta a\eta)^2} + (\theta-\tfrac{1}{2}-\theta I)
\frac{4a^2}{(1-\theta a\eta)^4}\\
  &=& 1 + 2 I a - 2a^2 + (20\theta^2-8\theta-8\theta I)a^3 + {\cal O}(a^4). 
\end{eqnarray*}
This yields
\[
|S_\theta (-2a,a \eta,a \eta)|^2  ~=~ 1 + (40\theta^2-16\theta)a^3 + {\cal O}(a^4).
\phantom{xxxxxxxxxxxxxx}
\]
The right-hand side is bounded by 1 for $a\uparrow 0$ only if 
$40\theta^2-16\theta \ge 0$.
Hence, it must hold that $\theta \geq \frac{2}{5}$.
\quad\mbox{\tiny $\blacksquare$}
\end{proof}
\vskip0.3cm\noindent
Based on strong numerical evidence (see Sect.~\ref{Conclusion}) we 
conjecture that the condition on $\theta$ in Theorem \ref{Thm3} is 
also sufficient, but a proof is currently lacking.

The above results dealt with real-valued $z_0$.
The two subsequent theorems concern arbitrary, complex-valued $z_0$.
A preliminary result is

\begin{lemma} \label{Le1}
Let $a, b, c \in \R$ be given. 
If\, $|a+b+c|=1$ and\, $|a\zeta^2+b\zeta+c| \leq 1$
whenever $\zeta \in \C$ with $|\zeta|=1$, then 
$ab+bc+4ac\geq 0$.
\end{lemma}
\begin{proof}
Consider the function $f$ defined by
\[
f(\varphi)=|ae^{2I\varphi}+be^{I\varphi}+c|^2 \quad (\varphi \in \R).
\]
There holds
\begin{equation*}
\begin{split}
f(\varphi)&=[a\cos(2\varphi)+b\cos(\varphi)+c]^2 +[a\sin(2\varphi)+b\sin(\varphi)]^2\\
&=a^2+b^2+c^2+2ab \cos(\varphi)+2bc\cos(\varphi)+2ac\cos(2\varphi).
\end{split}
\end{equation*}
One readily verifies that $f(0)=1$, $f^\prime (0)=0$, $f^{\prime\prime}(0)=-2(ab+bc+4ac)$ 
and hence
\[
f(\varphi) =1-(ab+bc+4ac)\varphi^2+\mathcal{O}(\varphi^3) \quad (\varphi \rightarrow 0).
\]
Using that $f(\varphi)\leq 1$ whenever $\varphi \in \R$, proves the assertion.
\quad\mbox{\tiny$\blacksquare$}
\end{proof}
\vskip0.3cm\noindent
For the case where $z_1, z_2$ are real-valued, we obtain the following necessary 
lower bound on $\theta$.
Numerical experiments indicate that this bound is sufficient as well. 

\begin{theorem}\label{Thm4}
Suppose $|S_\theta (z_0,z_1,z_2)| \leq 1$ for all $z_0 \in \C$ and $z_1, z_2\in \R$
satisfying $(\ref{cond})$.
Then $\theta \geq \frac{5}{12}$.
\end{theorem}

\begin{proof}
Setting $q=p^2+pz+(\frac{1}{2}-\theta)z^2$ and  $w=p+(1-\theta)z$, we can write
\begin{equation}\label{Sform}
S_\theta (z_0,z_1,z_2) = \frac{\tfrac{1}{2} z_0^2 + w z_0 + q}{p^2}\,.
\end{equation}
Let $y=2\sqrt{z_1 z_2}$.
Since \mbox{$|S_\theta (z_0,z_1,z_2)| \leq 1$} for all $z_0 \in \mathbb{C}$ 
with $|z_0| \leq y$ we have
\[
|\tfrac{1}{2} y^2 \zeta^2 + wy\zeta + q| \le p^2 \quad \textrm{for all } \zeta \in \C 
\textrm{ with } |\zeta|\le 1.
\]
Assume $z_1 = z_2$. Then $z=-y$ and it is easily seen that $\tfrac{1}{2} y^2 + wy + q = p^2$.
Therefore Lemma $\ref{Le1}$ can be applied and, using $y\ge 0$, this leads to the 
necessary condition 
\begin{equation}
\label{nec_con}
\tfrac{1}{2}wy^2 + qw + 2qy \geq 0 \quad \textrm{whenever } z_1 = z_2 \le 0.
\end{equation}
Denote $x = \theta y$. 
Then 
\[
p = 1 + x + \tfrac{1}{4} x^2.
\]
Next, after some computations, there follows
\[
\tfrac{1}{2}wy^2 + qw + 2qy = p^3+p^2x - \frac{1}{\theta}(x^3+2px^2).
\]
By (\ref{nec_con}), we arrive at
\[
\theta \ge \frac{x^3+2px^2}{p^3+p^2x}\,.
\]
The right-hand side is a rational function of $x\ge 0$, which is readily 
seen to have a global maximum at $x =2$.
Inserting this value yields the lower bound $\theta \ge \tfrac{5}{12}$.
\rightline{\mbox{\tiny$\blacksquare$}}
\end{proof}

The final result in this section concerns the most general case, where 
all $z_j$ are complex-valued.
To derive this result we employ a lemma from \cite{IHW07} pertinent to 
the condition $(\ref{cond})$.
For completeness, its concise proof is included here.
\begin{lemma}\label{Le2}
If $z_1, z_2\in \C$ with $\Re z_1 \le 0$, $\Re z_2 \le 0$, then
\[
2\sqrt{\Re z_1 \Re z_2} \le 
\left|\frac{p}{2\theta}\right| - \left|\frac{p}{2\theta}+z\right|.
\]
\end{lemma}
\begin{proof}
Define the vectors
\[
\mathbf{v}_j = 
\left(
  \begin{array}{c}
    \sqrt{-2\Re z_j}\\[5pt]
    \left|1+\theta z_j\right|/\sqrt{2\theta}\\
  \end{array}
\right), \quad j=1, 2.
\]
Their Euclidean norms are
\[ \|\mathbf{v}_j\|
=\sqrt{-2\Re z_j +\frac{|1+\theta z_j|^2}{2\theta}} \,=\,\frac{|1-\theta
z_j|}{\sqrt{2\theta}}.
\]
Next, their standard inner product is
\[
\langle \mathbf{v}_1 , \mathbf{v}_2 \rangle =
2\sqrt{\Re z_1\Re z_2} + \frac{|(1+\theta z_1)(1+\theta
z_2)|}{2\theta} =
2\sqrt{\Re z_1 \Re z_2} + \left|\frac{p}{2\theta}+z\right|.
\]
Applying the Cauchy--Schwarz inequality gives
\[
2\sqrt{\Re z_1 \Re z_2} + \left|\frac{p}{2\theta}+z\right|
\le \frac{|1-\theta z_1||1-\theta z_2|}{2\theta}
= \left|\frac{p}{2\theta}\right|,
\]
which concludes the proof.\quad\mbox{\tiny$\blacksquare$}
\end{proof}
\vskip0.3cm\noindent
For the most general case, we have the following positive result:
\begin{theorem}\label{Thm5}
If $\frac{1}{2}\leq\theta\leq1$, then $|S_\theta (z_0, z_1, z_2)| \leq1$ 
whenever $z_0, z_1, z_2\in \C$ satisfy $(\ref{cond})$.
\end{theorem}

\begin{proof}
The expression (\ref{Sform}) for $S_\theta $ yields 
\[
|S_\theta (z_0,z_1,z_2)| \leq \frac{1}{2}\left|\frac{z_0}{p}\right|^2 + 
\left|\frac{z_0}{p}\right| \left|1+(1-\theta)\frac{z}{p}\right| + 
\left|1+\frac{z}{p}+\left(\frac{1}{2}-\theta\right)\frac{z^2}{p^2}\right|.
\]
By invoking Lemma \ref{Le2} it follows that $|S_\theta (z_0,z_1,z_2)|$ is bounded 
from above by
\[
\frac{1}{2}\left( \frac{1}{2\theta} - \left|\frac{1}{2\theta}+\frac{z}{p}\right|\right)^2 
+\left( \frac{1}{2\theta} - \left|\frac{1}{2\theta}+\frac{z}{p}\right|\right) 
\left|1+(1-\theta)\frac{z}{p}\right| + \left|1+\frac{z}{p}+\left(\frac{1}{2}-\theta\right)\frac{z^2}{p^2}\right|.
\]
\vskip2pt\noindent
We can write
\[
1+2\theta \frac{z}{p}= re^{I\varphi} \quad 
\textrm{with } 0\leq r\leq 1 \textrm{ and } 0\leq \varphi < 2\pi.
\]
Define 
\begin{eqnarray*}
  f_1(\varphi,r) &=& \left|2\theta+(1-\theta)(re^{I\varphi}-1)\right|, \\
  f_2(\varphi,r) &=& \left|8\theta^2+4\theta(re^{I\varphi}-1)+ (1-2\theta)(re^{I\varphi}-1)^2\right|.
\end{eqnarray*}
Then it follows that
\begin{equation}\label{upS}
|S_\theta (z_0,z_1,z_2)| \le \frac{(1-r)^2 + 2(1-r)f_1(\varphi,r) + f_2(\varphi,r)}{8\theta^2}\,.
\end{equation}

Let $\frac{1}{2}\leq\theta\leq1$.
We prove that the right-hand side of (\ref{upS}) is bounded by $1$ for all $0\leq r\leq 1$, 
$0\leq \varphi < 2\pi$.
First note that $f_j(2\pi-\varphi,r) = f_j(\varphi,r)$ ($j=1, 2$) and therefore it suffices 
to consider $0\leq \varphi \le \pi$.
Let $r\in [0,1]$ be fixed but arbitrary and define $g_j(\varphi) = f_j(\varphi,r)^2$
($j=1, 2$).
For the function $g_1$ it is readily verified that
\[
g_1^\prime (\varphi) = -2(3\theta-1)(1-\theta) r \sin \varphi. 
\]
This directly implies that $g_1$, and hence $f_1$, is nonincreasing on $[0,\pi]$.
For the function $g_2$ a more elaborate computation shows
\[
g_2^\prime (\varphi) = 4(2\theta -1)(4\theta-1)
\left[ 2(2\theta -1) r \cos\varphi + r^2 - (4\theta-1) \right] r \sin\varphi.
\]
In view of 
\[
2(2\theta -1) r \cos\varphi + r^2 - (4\theta-1) \le 
2(2\theta -1) + 1 - (4\theta-1) = 0
\] 
we find that also $g_2$, and hence $f_2$, is nonincreasing on $[0,\pi]$.
Consequently, it is sufficient to prove that the right-hand side of (\ref{upS}) 
is bounded by $1$ whenever $0\leq r\leq 1$, $\varphi =0$. 
Write $s=r-1 \in [-1, 0]$. 
One easily verifies that
\begin{eqnarray*}
f_1(0,r) &=& 2\theta+(1-\theta)s, \\
f_2(0,r) &=& 8\theta^2+4\theta s + (1-2\theta) s^2.
\end{eqnarray*}
Inserting this and rearranging terms, it follows that the upper 
bound (\ref{upS}) is (in fact) equal to $1$ whenever $0\leq r\leq 1$, 
$\varphi =0$.
\quad\mbox{\tiny$\blacksquare$}
\end{proof}
\vskip0.3cm\noindent
Numerical evidence leads to the conjecture that the conclusion of Theorem 
\ref{Thm5} is valid for all $\theta \geq \frac{5}{12}$,\, i.e.,~under the 
(necessary) lower bound of Theorem \ref{Thm4}.
A proof of this does not appear to be straightforward.
We note that in the proof above the assumption $\frac{1}{2} \leq \theta \leq 1$ 
is used in an essential manner.

\setcounter{equation}{0}
\setcounter{theorem}{0}
\section{Application and discussion}\label{Conclusion}
In this section we discuss an application 
to convection-diffusion equations (\ref{PDE}).
We semi-discretize on the unit square $[0,1]\times [0,1]$ by using central 
second-order FD schemes on a Cartesian grid with mesh widths $\Delta x$ 
and $\Delta y$ in the $x$ and $y$ directions, respectively:
\begin{subeqnarray}\label{xy}
\left(u_x\right)_{i,j}
&\approx &  \frac{u_{i+1,j}-u_{i-1,j}}{2\Delta x}\\
\left(u_y\right)_{i,j}
&\approx & \frac{u_{i,j+1}-u_{i,j-1}}{2\Delta y}\\
\left(u_{xx}\right)_{i,j}
&\approx & \frac{u_{i+1,j}-2u_{i,j}+u_{i-1,j}}{(\Delta x)^2}\\
\left(u_{yy}\right)_{i,j}
&\approx & \frac{u_{i,j+1}-2u_{i,j}+u_{i,j-1}}{(\Delta y)^2}\\
\left(u_{xy}\right)_{i,j}
&\approx
&\frac{(1+\beta)(u_{i+1,j+1}+u_{i-1,j-1})-(1-\beta)(u_{i-1,j+1}+u_{i+1,j-1})}
{4\Delta x \Delta y}\notag\\
&&+\frac{4\beta u_{i,j}-2\beta(u_{i+1,j}+u_{i,j+1}+u_{i-1,j}+u_{i,j-1})}
{4\Delta x \Delta y}\,.
\end{subeqnarray}
Here $\beta$ denotes a real parameter with $-1\le \beta\le 1$ and we use
the notation $u_{i,j} = u(i \Delta x, j\Delta y, t)$.
We note that the right-hand side of (\ref{xy}e) is the most general form of a 
second-order FD approximation of the mixed derivative $u_{xy}$ based on a 
centered 9-point stencil.
When $\beta=0$, it reduces to the well-known 4-point formula
\[ \left(u_{xy}\right)_{i,j}
\approx \frac{u_{i+1,j+1}+u_{i-1,j-1}-u_{i-1,j+1}-u_{i+1,j-1}}{4\Delta x \Delta y}\,.
\]

Assuming constant coefficients and a periodic boundary condition for (\ref{PDE}), 
the above FD discretization yields a splitted, semi-discrete system (\ref{ODE}), 
(\ref{splitting}) where $F_j(t,v)=A_j v$ for $j=0,1,2$ with constant matrices $A_j$.
The matrix $A_0$ represents the cross derivative term in (\ref{PDE}) and $A_1$, 
$A_2$ represent the spatial derivatives in the $x$ and $y$ directions, 
respectively.
The periodicity condition implies that the $A_j$ are Kronecker products of 
circulant (thus normal) matrices that commute with each other, and are 
therefore simultaneously diagonalizable by a unitary matrix.
Hence, stability can be rigorously analyzed by considering the scalar test 
equation (\ref{test}) with $\lambda_j$ eigenvalues of $A_j$ ($0\le j\le 2$). 
This is equivalent to a von Neumann stability analysis.
By inserting discrete Fourier modes, it follows that the scaled eigenvalues 
$z_j$ are given by
\begin{subeqnarray}\label{zzz}
z_0&=&(d_{12}+d_{21})\,b\,[-\sin\phi_1\sin\phi_2+\beta(1-\cos\phi_1)(1-\cos\phi_2)]\,, \\
z_1&=&-2d_{11}a_1(1-\cos\phi_1)+Ic_1q_1\sin\phi_1\,, \\
z_2&=&-2d_{22}a_2(1-\cos\phi_2)+Ic_2q_2\sin\phi_2\,,
\end{subeqnarray}
\addtocounter{equation}{-1}
where
\[ 
a_1=\frac{\Delta t}{(\Delta x)^2}\,,\quad
a_2=\frac{\Delta t}{(\Delta y)^2}\,,\quad
b=\frac{\Delta t}{\Delta x\Delta y}\,, \quad
q_1=\frac{\Delta t}{\Delta x}\,,\quad
q_2=\frac{\Delta t}{\Delta y}\,.
\]
The angles $\phi_j$ are integer multiples of $2\pi/m_j$ ($j=1,2$) where
$m_1$, $m_2$ are the dimensions of the grid in the $x$ and $y$ directions, 
respectively.

Using the positive semi-definiteness of the diffusion matrix $D$, an elementary 
calculation shows \cite{IHW07} that $z_0, z_1, z_2$ fulfill the condition 
(\ref{cond}), independently of $\Delta t$, $\Delta x$, $\Delta y$.
Upon invoking Theorem \ref{Thm5} the following neat stability result is
obtained for the MCS scheme applied to (\ref{PDE}).

\begin{theorem}\label{th3.1}
Consider equation (\ref{PDE}) with positive semi-definite matrix $D$ and 
periodic boundary condition.
Let the semi-discrete system (\ref{ODE}), (\ref{splitting}) be obtained 
after FD discretization and splitting as described in this section.
Then the MCS scheme (\ref{MCS}) is unconditionally stable when applied 
to (\ref{ODE}), (\ref{splitting}) whenever $\frac{1}{2} \le \theta \le 1$.
Moreover, this conclusion remains valid when any other stable FD 
discretizations for $u_x$\,, $u_y$ are used in place of (\ref{xy}a), 
(\ref{xy}b).
\end{theorem}

\noindent 
The last part of Theorem \ref{th3.1} follows directly from the fact that 
the real parts of the new eigenvalues $z_1$, $z_2$ are always smaller 
than those of (\ref{zzz}b), (\ref{zzz}c), respectively, and hence,
(\ref{cond}) remains true.

An inspection of (\ref{zzz}a) yields that the eigenvalues $z_0$ have the 
property that their imaginary part is identically equal to zero.
Accordingly, it is of particular interest to know all parameter values $\theta$ 
such that the stability requirement $(\ref{cond}) \Rightarrow (\ref{stabcondS})$ 
holds for just real-valued $z_0$.
Theorem \ref{Thm5} provides the sufficient condition $\frac{1}{2} \le \theta \le 1$,
whereas Theorem \ref{Thm3} yields the necessary condition $\theta \geq \frac{2}{5}$.

Next, we remark that the MCS scheme has recently been applied successfully in \cite{IHF10} 
to actual convection-diffusion equations (\ref{PDE}) with mixed derivative terms using 
the parameter value $\theta = \frac{1}{3}$.
This seems to be surprising, as this value was determined \cite{IHW09} for pure 
diffusion equations (\ref{PDE}) and it clearly does not satisfy the necessary 
condition $\theta \geq \frac{2}{5}$ for equations with convection.
We note that reasons for choosing a smaller $\theta$ in the MCS scheme are a 
reduced error constant and better damping properties compared to the original CS 
scheme, see \cite{IHF10}.
\begin{figure}
\begin{center}
\includegraphics[width=0.95\textwidth]{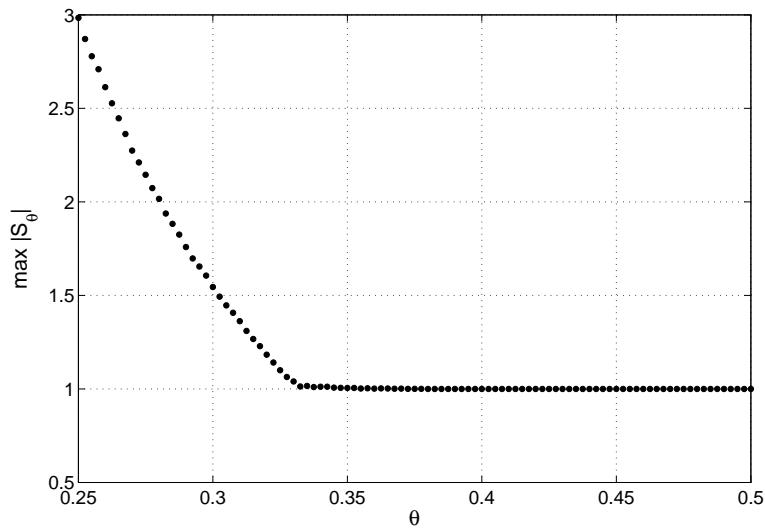}\\
\caption{Estimated maximum of $|S_\theta(z_0,z_1,z_2)|$ 
under (\ref{cond}) with $z_0\in \R$.}
\label{maxabs}
\end{center}
\end{figure}

Theoretical results on the latter two issues are not known at this moment.
To gain insight, we have performed a numerical experiment.
Let $r_{1,0}$ and $r_{i,j}$ for $i,j=1,2$ denote independent, uniformly 
distributed random numbers in $[0,1]$ and consider random triplets
$(z_0,z_1,z_2)$ given by
\begin{equation*}
z_0 = (2r_{1,0}-1)\cdot2\sqrt{\re z_1\re z_2} ~~~{\rm and}~~~
z_j = -10^{1-5r_{1,j}}\pm I\,10^{1-5r_{2,j}}~~(j=1,2).
\end{equation*}
Then (\ref{cond}) holds and $z_0\in \R$.
For each $\theta= \frac{1}{4} + \frac{k}{400}$ with $k=0,1,\ldots,100$ 
we computed the maximum value of $|S_\theta(z_0,z_1,z_2)|$ over two 
million points $(z_0,z_1,z_2)$ above.
The outcome is displayed in Figure \ref{maxabs}.

Figure \ref{maxabs} reveals the intriguing result that the estimated 
maximum value of $|S_\theta|$ is very close to $1$ whenever 
$\theta \ge \frac{1}{3}$.
For $\theta =  \frac{1}{3}$ we arrive at a maximum value of $1.02$.
Additional experiments in this case suggest that $|S_\theta|$ is larger 
than~1 for a limited set of points $(z_0,z_1,z_2)$, and at most~1 under 
only a slightly stronger condition on $z_0$ than in (\ref{cond}).
Because of these observations, it is very plausible that the MCS scheme 
performs well in actual applications to (\ref{PDE}), also with convection, 
already when $\theta = \frac{1}{3}$. 

Subsequently, an examination of the obtained numerical results indicates 
that $|S_\theta|\le 1$ for all $\theta \geq \frac{2}{5}$.
This supports our conjecture formulated below Theorem \ref{Thm3}.

In view of the above, it is likely that the condition on $\theta$ in 
Theorem \ref{th3.1} can be relaxed to $\theta \geq \frac{2}{5}$, and
next, that a slightly modified version of Theorem \ref{th3.1} holds
under the (weaker) assumption $\theta \ge \frac{1}{3}$.
In future research we intend to study these issues theoretically.

\section*{Acknowledgments}
This work has been supported financially by the Research Foundation--Flanders, 
FWO contract no. G.0125.08.

\end{document}